\documentclass[11pt]{amsart}
\usepackage[utf8]{inputenc}
\usepackage{amsfonts}
\usepackage{mathtools}
\usepackage{amsthm}
\usepackage{hyperref}
\usepackage{xcolor}
\usepackage{blkarray}
\usepackage{csquotes}
\usepackage{lineno}
\theoremstyle{plain}
\newtheorem{theorem}{Theorem}[section]

\newtheorem{lemma}{Lemma}[section]
\newtheorem{proposition}{Proposition}[section]

\newtheorem{remark}{Remark}[section]

\newtheorem{definition}{Definition}[section]
\newcommand{\rstr}[2]{\left. #1 \right\rvert_{#2}}

\newcommand{\abs}[1]{\left\lvert#1\right\rvert}
\newcommand{\CAT}{\mathrm{CAT}(-1)}
\DeclareMathOperator{\lex}{ShortLex}

\title{Orbital Counting in Conjugacy Classes}
\author{Alexander Baumgartner and Mark Pollicott}
\date{\today}

\newcommand{\Addresses}{{% additional braces for segregating \footnotesize
  \bigskip
  \footnotesize

  A.~Baumgartner, \textsc{SNS, Piazza dei Cavalieri, 7 - 56126 Pisa}\par\nopagebreak
  \textit{E-mail address}, A.~Baumgartner: \texttt{alexander.baumgartner@sns.it}

  \medskip

  M.~Pollicott, \textsc{Department of Mathematics, Warwick University, Coventry, CV4 7AL}\par\nopagebreak
  \textit{E-mail address}, M.~Pollicott: \texttt{masdbl@warwick.ac.uk}

  \medskip

}}

\usepackage[
  backend=biber,
  style=alphabetic,
  sorting=nyt,
  sortcites=true,
  maxbibnames=99,
  maxcitenames=4,
  maxalphanames=4,
  minalphanames=3,
  giveninits=true,
  doi=false,
  isbn=false,
  url=false,
  eprint=true,
  backref=false
]{biblatex}
\renewbibmacro{in:}{%
  \ifentrytype{article}
    {}
    {\bibstring{in}%
     \printunit{\intitlepunct}}}

\begin{document}
\begin{abstract}
In this article we consider a restricted  orbital counting problem for the action of certain discrete groups on suitable spaces. In particular, we present asymptotics for counting those points in an orbit restricted to  a single conjugacy class.  A classical example would be cocompact actions of a discrete group acting isometrically on a simply connected manifold with pinched negative curvature. More generally, we obtain results for convex cocompact actions on $CAT(-1)$ spaces.   

\end{abstract}
\maketitle
\section{Introduction}
In this note we will consider the asymptotics of 
restricted orbital counting  for orbits of groups acting on suitable spaces.
To  illustrate  this, consider the original  case where $X$ is a complete simply connected Riemannian manifold with pinched sectional curvature, and let $\Gamma$ be a discrete group acting isometrically on $X$.
A classical problem, which is the hyperbolic analogue of the Gauss circle problem, is the study of asymptotics of the function
\begin{equation}\label{equation: orbit counting}
    \# \{g\in\Gamma: d(x_0,g \cdot x_0)\leq T\}
\end{equation}
for some fixed $x_0\in X$ as $T\to \infty$. Under certain conditions on the quotient space $M:=X/\Gamma$, we have the following.
\begin{proposition}[after Huber and Margulis]\label{theorem: Margulis}
 There exists $C = C(x_0)>0$ such that 
\begin{equation}\label{equation: Margulis}
    \# \{g\in\Gamma: d(x_0,g \cdot x_0)\leq T\} \sim Ce^{\delta T}
\end{equation}
as $T \to +\infty$.
\end{proposition}

Huber's \cite{Huber59} proof holds when $M$ is an orientable compact manifold of constant negative sectional curvature and Margulis' \cite{Margulis69} proof holds when $M$ is a compact manifold without boundary with negative sectional curvature, but it has since been greatly generalised by various authors. Furthermore, in many settings one can obtain error terms. See, for example, \cite{Huber59,Patterson75,Sullivan79,LaxPhillips82,CdV83,Patterson,Lalley,DalboPeigne98,PolSharp94,PolSharp98,Cantrell25} for related results.

The main theorem in this article is a restricted counting problem where instead of counting the images under $\Gamma$,  we consider images under elements in  a fixed conjugacy class $\mathrm{Cl}(g):= \{h^{-1}gh: h\in\Gamma \}$ with $g\in \Gamma$ and obtain the following. 
  
\begin{theorem}\label{Hypothesis:First Step}
Let $\Gamma$ be a nonelementary word-hyperbolic group acting properly discontinuously, isometrically and convex cocompactly on a proper \linebreak $\CAT$
space $X$. Assume the geodesic flow on $X/\Gamma$ is mixing. Let $g \in \Gamma$ such that $\mathrm{Cl}(g)$ is not finite and let $x_0 \in X$.
Then there exists $C = C(x_0,g)>0$
% and $h>0$
 such that 
\begin{equation}
    \# \{g' \in \mathrm{Cl}(g) \hbox{ : }  d(x_0,g' \cdot x_0)\leq T\} \sim Ce^{\delta T/2}
\end{equation}
as $T\to +\infty$, where $\delta$ is the critical exponent of $\Gamma$ for its action on $X$.
\end{theorem}
Note that the exponential growth rate in this restricted case 
(Theorem \ref{Hypothesis:First Step})
is half that in the unrestricted case (Proposition \ref{theorem: Margulis}). We refer to the Appendix for the definition of a $\CAT$-space, noting that examples include simply connected Riemannian manifolds of sectional curvature $\leq -1$ and metric trees.
Theorem \ref{Hypothesis:First Step} is already known in certain settings where $\Gamma$ is a group acting isometrically on a space $X$ which is either a complete simply connected Riemannian manifold of pinched negative sectional curvature, or a tree. Huber \cite{Huber1} first established the theorem in the case where $X/\Gamma$ is a compact manifold of constant negative curvature and later obtained error terms in the same setting \cite{Huber3}. An analogous result for $(q+1)$-regular trees with $q$ odd was obtained in \cite{Douma} using spectral methods for the discrete Laplacian. Huber also interpreted the counting result in geometric terms as counting geodesic arcs perpendicular to certain quasiconvex subsets of $X$. Parkkonen and Paulin further developed this interpretation in \cite{Parkkonen2015}, and used results from their previous work \cite{PakkonenPaulin14} in counting perpendicular arcs to obtain exact asymptotics in the case where $\Gamma$ is geometrically finite and $X$ is the hyperbolic plane, as well as bounds and weaker estimates in the higher-dimensional, variable negative curvature case. These methods were further extended by Broise-Alamichel-Parkkonen-Paulin \cite{broisealamichel19} who obtain asymptotics in the case where $X$ is a metric or simplicial tree, and Honaryar \cite{honaryar22}, who obtains asymptotics where $X/\Gamma$ is a compact manifold with a pinching condition on the curvature. These methods rely on counting results developed in \cite{OS13,PakkonenPaulin14} for counting perpendicular arcs in $X/\Gamma$ between projections of closed convex subsets of $X$.

In this manuscript, we adopt a different approach using thermodynamic formalism. Such methods were previously employed by Kenison and Sharp \cite{KenisonSharp17,KenisonSharp} who proved asymptotics in the case where $\Gamma$ is a free group acting on a metric tree with non-arithmetic length spectrum, as well as a central limit theorem for free groups acting on $\CAT$-spaces. It was also employed by the second author \cite{Pollicott}, who sketched Theorem \ref{Hypothesis:First Step}, restricted to the  case where $\Gamma$ is a free group. In that case the  restriction was due to the fact that it was previously unclear whether a coding scheme existed that allowed us to enumerate all the elements $\mathrm{Cl}(g)$ for a general hyperbolic group.

The  main innovation in the present paper is to employ a coding due to Redfern \cite{Redfern93}, see also \cite{Holt99}, which uniquely enumerates the right-cosets
$Z(g)h$ of the centraliser \[Z(g):=\{h\in \Gamma: hg=gh\}\] with a representative of minimal word length. Since $Z(g)a= Z(g)b$ for $a,b\in \Gamma$ if and only if $a^{-1}ga = b^{-1}gb$, we have that $ba^{-1}\in Z(g)$ and therefore that each right coset corresponds to a unique element of $\mathrm{Cl}(g)$. We thus also obtain an enumeration of $\mathrm{Cl}(g)$. One advantage of using a coding which works at the group-theoretic level is that it allows us to treat the counting problem for convex cocompact actions on $\CAT$-spaces without requiring an improvement in the geometric estimates in \cite{honaryar21}. In particular, Theorem \ref{Hypothesis:First Step} is new when $X/\Gamma$ is an infinite-volume convex cocompact Riemannian manifold of either variable negative curvature or of dimension greater than $3$. The theorem also recovers the asymptotics in \cite{Parkkonen2015,honaryar22,KenisonSharp17,KenisonSharp} and Theorem 13.1 (1) in \cite{broisealamichel19}.

 \begin{remark}
 For the case of discrete groups acting convex cocompactly on $\CAT$-spaces, Roblin \cite{Roblin03} showed that obtaining an asymptotic of the form \eqref{equation: Margulis} is equivalent to having non-arithmetic length spectrum. Here, the length spectrum is defined for each conjugacy class $\mathrm{Cl}(g)$ of a hyperbolic element $g$ by 
 \[\mathrm{length}(\mathrm{Cl}(g)):= \inf_{x\in X}d(x,g\cdot x).\]
 In the case where $\Gamma$ has no elliptic elements, this corresponds to the length of the closed geodesic associated to $\mathrm{Cl}(g)$. Non-arithmeticity of the length spectrum here means that the subgroup generated by the lengths of all hyperbolic conjugacy classes is dense in $\mathbb{R}$. 
 
\end{remark}

\section{Preliminaries and Notation}
In this section, we define some notation and recall some facts about hyperbolic groups.

We shall later consider directed graphs $\mathcal{G}$ whose paths starting at some distinguished vertex $*$ generate representations of group elements of shortest length. It is therefore important that we distinguish between the paths in $\mathcal{G}$, the representations of group elements, and the group elements themselves. 

\subsection{Shortest Representations}\label{subsection: Shortest Representations}

Let $\Gamma$ be a finitely generated group and let ${\Gamma_0}\subset \Gamma$ be a finite set of generators. Assume furthermore that ${\Gamma_0}$ is symmetric, i.e. $\overline{a}\in {\Gamma_0}$ implies that $\overline{a}^{-1}\in {\Gamma_0}$. Assume that the cardinality $\abs{{\Gamma_0}}$ of ${\Gamma_0}$ is equal to $n$. We define $S:=\{1,2,\ldots, n\}$ and fix some bijection $S\to {\Gamma_0}: a\mapsto \overline{a}$. We define the set of \textbf{words}
\[S^*:= \bigcup_{k\in\mathbb{N}} S^k,\]
where $S^k$ is the set-theoretic $k$-th Cartesian product. There is a natural map from the set of words in $S^*$ to $\Gamma$ given as follows. For a word $w=w_1w_2\cdots w_k\in S^*$ with $w_i\in S$ for all $i\in \{1,2,\ldots,k\}$, we define 
\begin{equation}\label{equation: words to elements}
  \overline{w}=\overline{w}_1 \overline{w}_2 \cdots \overline{w}_k.  
\end{equation}
 There is also a natural interpretation of $S^*$ as the set of paths in the \textbf{Cayley Graph} $(\Gamma,{\Gamma_0})$ starting at the identity element $e\in \Gamma$. Indeed, we may identify the word $w=w_1w_2\ldots w_k\in S^*$ with the path in $(\Gamma,{\Gamma_0})$  successively connecting the sequence of vertices
\[(e,\quad \overline{w}_1,\quad  \overline{w}_{1}\overline{w}_2\quad, \ldots,\quad  \overline{w}_1\cdots \overline{w}_{k-1}\overline{w}_k).\]
We define the length $\abs{w}$ of a word $w\in S^*$ to be the unique value of $k$ for which $w\in S^k$, and we define for a group element $g\in \Gamma$ the \textbf{word-length} 
\[\abs{g}_{{\Gamma_0}}:= \min\{\abs{w}: w\in S^* \text{ and } \overline{w}=g\}.\]
A word $w\in S^*$ such that $\abs{w}= \abs{\overline{w}}_{{\Gamma_0}}$ is called \textbf{geodesic}, as it corresponds to a geodesic path starting from $e$ in $(\Gamma,{\Gamma_0})$ equipped with the word-length metric.

\subsection{The Redfern Coding}
This section briefly introduces the \textbf{Redfern coding} for right cosets of a quasiconvex subgroup $H\subset \Gamma$. The case $H=\{e\}$ is often called the \textbf{Cannon coding}. The existence of a coding scheme for groups acting cocompactly on $\mathbf{H}^d$ is due to Cannon \cite{Cannon84}. That this generalises to word-hyperbolic groups was remarked upon by Gromov \cite{Gromov87}. A detailed proof for word-hyperbolic groups is given in e.g. Ghys and de la Harpe in \cite{Ghys90}.
\begin{definition}\label{definition: path-labelling}
    Let $S=\{1,2,\ldots,n\}$ for some $n\in \mathbb{N}$ and let $\mathcal{G}$ be a finite directed graph with
    \begin{itemize}
        \item a vertex set $V$,
        \item an edge set $E\subset V\times V$,
        \item a distinguished vertex $*$ such that no edge in $E$ ends at $*$ and
        \item an edge labelling map $\lambda:E\to S:u\mapsto \lambda(u)$.
    \end{itemize}
    Let $\mathfrak{L}(\mathcal{G})$ be the set of paths $v=(*,v_1,v_2,\ldots,v_l)$, $l\in\mathbb{N}$ in $\mathcal{G}$ starting at the vertex $*$, i.e. the set of finite sequences $v$ starting at $v_0=*$ and satisfying
    \[ (v_{i-1},v_i)\in E\]
    for all $i\in \{1,\ldots,l\}$.
    We extend $\lambda$ to a \textbf{path-labelling map} by defining $\lambda: \mathfrak{L}(\mathcal{G}) \to S^*$ by setting
    \[\lambda\left(v\right):=\lambda((*,v_1))\lambda((v_1,v_2))\cdots \lambda((v_{l-1},v_l)).\]
\end{definition}
For later consistency, we use the convention that the \enquote{path} $(*)$ consisting of only the starting vertex is in $\mathfrak{L}(\mathcal{G})$ and that $\lambda((*))=\emptyset$, with $\overline{\emptyset}=e$.
Assuming $S$ is in some bijection with a set of generators ${\Gamma_0}$ of $\Gamma$, as in Subsection \ref{subsection: Shortest Representations}, we see that the map $\lambda$ assigns paths in $\mathcal{G}$ starting at $*$ to words in $S^*$ and hence to representations of elements of $\Gamma$. 

The existence of a Cannon coding for hyperbolic groups essentially states that we can find a $(\mathcal{G},*,\lambda)$ as in Definition \ref{definition: path-labelling} such that each element of $\Gamma$ is represented uniquely with a geodesic word. We consider a related concept, where $\mathfrak{L}(\mathcal{G})$ maps onto unique geodesic representatives of right cosets of quasiconvex subgroups.

\begin{definition}\label{definition: strongly Markov}
    We say the group $\Gamma$ is $H$-\textbf{strongly Markov} for some subgroup $H\subset\Gamma $ if for any symmetric set of generators ${\Gamma_0}$, there exists $(\mathcal{G},*,\lambda)$ as in Definition \ref{definition: path-labelling} such that the associated path-labelling map $\lambda$ satisfies the following: 
    \begin{itemize}
        \item the image $\lambda(\mathfrak{L}(\mathcal{G}))$ consists only of geodesic words, and 
        \item the induced map 
    \[\pi_H:\mathfrak{L}(\mathcal{G})\to H\backslash\Gamma: v\mapsto H\overline{\lambda(v)}\]
    is bijective.
    \end{itemize}
    \end{definition}
 We shall henceforth always assume that given $(\mathcal{G},*,\lambda)$ as in the above definition, we may find a path from $*$ to any vertex in $V$, as such a graph can always be obtained by removing these inaccessible vertices without changing $\mathfrak{L}(\mathcal{G})$.
 \begin{definition}\label{definition: Redfern coding}
    We call $(\mathcal{G},*,\lambda)$ as in Definition \ref{definition: strongly Markov} with the above assumption a \textbf{Redfern coding}.
\end{definition}

\begin{remark}
    When $H= \{e\}$, Definition \ref{definition: strongly Markov} corresponds to the definition of \enquote{fortement Markov} in \cite{Ghys90}. We have presented this definition in a slightly different manner to the way it was presented in \cite{Ghys90}, in which the authors directly consider the map $\pi_{\{e\}}:\mathfrak{L}(\mathcal{G})\to \Gamma$ instead of factoring it through the set of words $S^*$, but our formulation is equivalent. 
\end{remark}
Let us now recall that if $G$ is a Gromov hyperbolic space, we say a subset $H\subset G$ is $C$-quasiconvex if for any geodesic $\gamma:[0,b]\to G$ with both endpoints in $H$ we have that $d_G(\gamma(t),H)\leq C$ for all $t\in [0,b]$, where $d_G$ is the metric on $G$. 
\begin{definition}
    If $\Gamma$ is a hyperbolic group, we say a subgroup $H$ of $\Gamma$ is \textbf{quasiconvex} if there exists some $C$ such that $H$ is $C$-quasiconvex with respect to the word-length metric on the Cayley graph $(\Gamma,{\Gamma_0})$.
\end{definition}
 This definition is independent of the choice of generators. With this definition in mind, we have the following.
\begin{theorem}[Redfern]\label{theorem: Redfern}
    If $\Gamma$ is a word-hyperbolic group and $H\subset \Gamma$ is a quasiconvex subgroup, then $\Gamma$ is $H$-strongly Markov.
\end{theorem}
We remark that we do not employ Redfern's original formulation \cite{Redfern93} here. We expand on this in somewhat more detail in the appendix.
\section{labelling Representatives of Cosets}\label{section: Counting Representatives of Cosets}
The proof of Theorem \ref{Hypothesis:First Step} is in essence analogous to the second author's unpublished proof when $\Gamma$ is a free group in \cite{Pollicott}, which involves counting orbit points $g'\cdot x_0$ subject to an additional restriction on the shortest length representation of $g'$. We shall do the same, except we use a Redfern coding that allows us to enumerate unique representatives of right cosets of the centraliser $Z(g):=\{h\in \Gamma: hg=gh\}$. This restriction allows us to avoid overcounting in the final proof, as we observed earlier that two group elements $a,b\in \Gamma$ satisfy $a^{-1}ga=b^{-1}gb$ if and only if $Z(g)a=Z(g)b$. However, since the methods in this section work for any quasiconvex subgroup $H\subset \Gamma$ of infinite index, we work in the more general setting.
\subsection{Subshifts of Finite Type}\label{subsection:Subshifts of finite type}
We first briefly show how, given a Redfern coding $(\mathcal{G},*,\lambda)$, we may embed $\mathfrak{L}(\mathcal{G})$ in a subshift of finite type. Following the construction in \cite{Pollicott1998ComparisonTA}, we add the vertex $0$ to enlarge the vertex set $V$ to $\tilde{V}:=\{0\}\cup V$.  We enlarge $E$ to $\widetilde{E}$ by adding a directed edge from every vertex in $\{0\}\cup V$ to $0$. Let $\tilde{\mathcal{G}}$ be the associated enlarged directed graph. We associate to $\tilde{\mathcal{G}}$ a transition matrix $A\in\tilde{V}\times \tilde{V}$ by letting $A(v_1,v_2)=1$ if and only if there is a directed edge in $\widetilde{E}$ from $v_1$ to $v_2$. We then define the shift space 
\[\Sigma := \{v=(v_i)_{i=0}^{\infty}:v_i\in \tilde{V}, A(v_i,v_{i+1})= 1 \text{ for all } i\in \mathbb{N}  \},\]
equipped with the left shift map $\sigma:\Sigma\to \Sigma: (v_0,v_1,\ldots)\mapsto (v_1,v_2,\ldots)$. We define a metric on $\Sigma$ by defining
the distance between $v=(v_i)_{i=0}^{+\infty} \in \Sigma$ and $v'=(v'_i)_{i=0}^{+\infty}\in\Sigma $ by \begin{equation}\label{Equation: the Natural Subshift Metric}
d_\Sigma(v,v') = 2^{-N}\text{, where } N = \inf \{k\in\mathbb{N}:v_k\neq v'_k\},\end{equation}
with the convention that $d_\Sigma(v,v)=0$.
It follows immediately that $\mathfrak{L}(\mathcal{G})$ is in bijection with all sequences of the form $(*,v_1,\ldots,v_k,0,0,\ldots)$. It will turn out useful to give an interpretation to sequences that do not start at $*$. 

\begin{definition}
    Let $\mathcal{P}(\mathcal{G})$ be the set of \textit{all} finite sequences of the form
\[(v_0,\ldots,v_k)\text{ with } v_i\in V \text{ for } i\in \{0,1,\ldots, k\} \text{ and } A(v_{i-1},v_{i})=1\]
for $i\in \{1,\ldots,k\}$.
\end{definition}
If we define the \textbf{terminating sequences} $\Sigma^{term}$ to be the set  $(v_k)_k\in\Sigma$ with $v_k=0$ for $k$ sufficiently large, we see that the map $\Sigma^{term}\to \mathcal{P}(\mathcal{G})$ obtained by dropping the zeros from the sequence is a bijection, and we may extend the map $\lambda: \mathfrak{L}(\mathcal{G}) \to S^*$ to a map $\mathcal{P}(\mathcal{G}) \to S^*$ which we also denote by \enquote{$\lambda$} by setting for $v=(v_0,\ldots,v_l)$:
    \[\lambda\left(v\right):=\lambda((v_0,v_1))\lambda((v_1,v_2))\cdots \lambda((v_{l-1},v_l)).\]

\begin{lemma}\label{lemma: all words are geodesic}
     Let $(\mathcal{G},*,\lambda)$ be a Redfern coding and $\Sigma$ its associated subshift of finite type. 
    The image under the map \[ \Sigma^{term}\to S^*:v\mapsto \lambda\left(v\right)\]
    consists only of geodesic words.
\end{lemma} 
\begin{proof}
 Assume by contradiction that there exists some terminating sequence $(v_0,v_1,\ldots,v_k,0,0,\ldots)\in \Sigma^{term}$ such that $\lambda(v_0,v_1)\cdots \lambda(v_{k-1},v_k)$ is not a geodesic word, i.e. assume that there is some $w\in S^*$ such that $\abs{w}<k$ and $\overline{w}= \overline{\lambda(v_0,v_1)\cdots \lambda(v_{k-1},v_k)}$. Denote $v=(v_0,\ldots,v_k)$ and let $v'= \left(v'_0,\ldots,v'_{k'}\right)\in \mathfrak{L}(\mathcal{G})$ with $v'_0=*$ and $v'_{k'}=v_0$. If we denote by $v'v$ the concatenated path $(v'_0,\ldots,v'_{k'},v_1,\ldots,v_k)$, then it is immediate that $v'v\in \mathfrak{L}(\mathcal{G})$ and hence that $\lambda(v'v)=\lambda(v')\lambda(v)$ is a geodesic word. However, we see that $\lambda(v')w$ is strictly shorter as a sequence than $\lambda(v'v)$, which is a contradiction.
\end{proof}
In general, the combinatorial structure of the transition matrix $A$ might be somewhat complicated. Indeed, this is an issue even for the Cannon coding. It will therefore be useful to consider its irreducible components separately. Recall that a square matrix $B\in \{0,1\}^{k\times k}$ is \textbf{irreducible} if for each pair of natural numbers $1\leq i,j\leq k$ we have that $B^m(i,j)>0$ for some $m\in \mathbb{N}$, and that it is aperiodic if there is some $m\in \mathbb{N}$ such that each entry in $B^m$ is strictly positive. Following e.g. Section 1.2 in \cite{Seneta}, we can assign an order to the vertex set $V\cup \{0\}$ so $A$ is of the form 
\begin{equation}\label{equation: Bs}
    A=\begin{pmatrix}
        B_{11} & 0 & \cdots &0\\
        B_{21} & B_{22} & \cdots &0\\
        \vdots & \vdots & \ddots & \vdots\\
        B_{k1} & B_{k2} &\cdots & B_{kk}
     \end{pmatrix}
\end{equation}
for some $k\in\mathbb{N}$ and so $B_{11}, \ldots ,B_{kk} $ are irreducible square matrices. Denote the subsets of $\tilde{V}$ on which the matrices $B_{ll}$ are respectively defined by $V_l$. Define the associated shift spaces to be the subset of sequences in $\Sigma$ consisting only of vertices in $V_l$.
\begin{definition}\label{definition: components graph}
    We define the \textbf{component graph} of $\widetilde{\mathcal{G}}$ to be the directed graph with vertex set $\{V_1,\ldots,V_k\}$ and a directed edge from $V_i$ to $V_j$ if and only if $B_{ij}$ is not the zero matrix, i.e. there exists an edge in $\widetilde{E}$ from a vertex in $V_i$ to a vertex in $V_j$. We define a partial order on the components $V_i$ by setting $V_i\prec V_j$ if there exists a path from $V_i$ to $V_j$ in the components graph.
\end{definition} 

\subsection{Hölder Continuity of the Potential}
We define a Hölder continuous potential related to the coset counting problem. The existence of one is well-established for the Cannon coding, which carries over in our case. We shall expand on this in some detail to emphasize the fact that Hölder continuity follows from convex cocompactness of the action of $\Gamma$ and Lemma \ref{lemma: all words are geodesic}, but we stress that the proof is essentially contained in Proposition 3 in \cite{PolSharp01} in the variable curvature case. Let us recall some standard concepts in hyperbolic geometry. We refer to e.g. \cite{Gromov87, Ghys90,Coornaert90} for more information.
Let $(Y,d_Y)$ be a metric space. Then for $x,y,z\in Y$, we define the Gromov product
\begin{equation}\label{equation: Gromov product}
      (y,z)_x = \frac{1}{2}\left(d_Y(x,y)+d_Y(x,z)-d_Y(y,z)\right),
\end{equation}
which behaves \enquote{nicely} under quasi-isometries.
\begin{definition}
    We say a map $f:(Y,d_Y)\to (Y',d_{Y'})$ is a $(\lambda,c)$\textbf{-quasi-isometry} if for all $y,z\in Y$,
\[\lambda^{-1}d_Y(y,z)-c \leq d_{Y'}(f(y),f(z))\leq \lambda d_Y(y,z)+c\]
and we say $f$ is a quasi-isometry if the above holds for some $\lambda,c>0$.
\end{definition}
If we have a $(\lambda,c)$-quasi-isometry, then there exist constants $B$ depending on $\lambda,c$ and the constants of hyperbolicity of $Y,Y'$ such that for all $x,y,z\in Y$:
        \begin{equation}\label{equation: quasi-invariance}
            \frac{1}{\lambda}(y,z)_x - B \leq (f(y),f(z))_{f(x)} \leq \lambda (y,z)_x+B.
        \end{equation}
A useful characterisation of convex cocompactness is the following.
\begin{lemma}
    The action of a discrete group $\Gamma$ on $X$ by isometries is convex cocompact if and only if for some (and hence any) finite set of generators ${\Gamma_0}$, the map $\Gamma\to X$ is a quasi-isometry with respect to the word metric on $\Gamma$.
\end{lemma}
We refer to the Main Theorem in \cite{Swenson} for the equivalence of this characterisation with other possible definitions. 

We use the following property of $\CAT$-spaces (Theorem 5.1 in \cite{NSpak}).
\begin{lemma}\label{definition: strong hyperbolicity}
    There exist constants $L,R_0>0$ such that for all $x,y,z,t\in X$, we have that $d(x,y)+d(z,t)-d(x,z)-d(y,t)=: R\geq R_0$ implies that \[\abs{d(x,y)+d(z,t)-d(x,t)-d(z,y)}\leq Le^{- R}.\]
\end{lemma}

\begin{proposition}\label{proposition: displacements and the Cayley graph} 
    Consider for $a\in {\Gamma_0}$ the function
    \[ \Psi_a: \Gamma\to \mathbb{R}: g\mapsto d(g\cdot x_0, x_0)-d(g\cdot x_0,a\cdot x_0).\]
    There exist constants $B>0$, $0<\theta<1$ for which 
    \[ \abs{\Psi_a(g)-\Psi_a(h)} \leq B\theta^{(g,h)_{e}}\]
    for all $g,h\in \Gamma$, where $(\cdot,\cdot)_\cdot$ is the Gromov product on the Cayley graph $(\Gamma,{\Gamma_0})$ with respect to the word metric. 
\end{proposition}
\begin{proof}
Let $d_{{\Gamma_0}}$ denote the word metric on $\Gamma$. Then the map $(\Gamma,d_{{\Gamma_0}})\to X:g\mapsto g\cdot x_0$ is a $(\lambda,c)$-quasi-isometry for some $\lambda, c$. If we let $(\cdot,\cdot)_{\cdot}$ and $\langle \cdot,\cdot\rangle_\cdot$ denote the Gromov product on $(\Gamma,d_{{\Gamma_0}})$ and $X$ respectively, we see by the triangle inequality that $ d(h\cdot x_0 ,a\cdot x_0)+d(  x_0,a\cdot x_0)\geq d(h\cdot x_0,  x_0)$, so by \eqref{equation: quasi-invariance} we have that 
    \begin{equation*}
        \begin{split}
            &d(g\cdot x_0,  x_0)+d(h \cdot x_0,a\cdot x_0)-d(g\cdot x_0,h\cdot x_0)-d(  x_0,a\cdot x_0) \geq \\ &2\langle g\cdot x_0,h\cdot x_0 \rangle_{x_0}-2d(x_0,a\cdot x_0)\geq 2(\lambda^{-1}(g,h)_e- B)
        \end{split}
    \end{equation*}
    for some uniformly bounded constant $B>0$. We obtain this proposition using Lemma \ref{definition: strong hyperbolicity}.
\end{proof}
As a direct application, we obtain our potential.
\begin{proposition}\label{proposition: existence of roof function}
    There exists some $0<\alpha<1$ and an $\alpha$-Hölder continuous function $r:\Sigma\to \mathbb{R}$ such that if $v=(*,v_1,v_2,\ldots,v_l,0,0,\ldots)\in \Sigma^{term}$, then
    \[r^l(v):= \sum_{k=0}^{l-1}r(\sigma^k v)  = d\left(\overline{\lambda\left(*,v_1,\cdots,v_l\right)}\cdot x_0 ,  x_0\right).\]
\end{proposition}
\begin{proof}
We define the function $r$ initially on terminating sequences by setting $r(v_0,0,0,\ldots)=0$ for all $v_0\in V\cup\{0\}$,  $r(v_0,v_1,0,\ldots)=d(\overline{\lambda(v_0,v_1)}\cdot x_0, x_0)$ for all $(v_0,v_1)\in E$, and
     \begin{equation*}
        \begin{split}
            r(v_0,v_1,\ldots,v_l,0,0,\ldots)=d\left(\overline{\lambda(v_0,v_1)}\cdots \overline{\lambda(v_{l-1},v_l)} \cdot x_0, x_0\right)\\-d\left(\overline{\lambda(v_1,v_2)}\cdots \overline{\lambda(v_{l-1},v_l)}x_0 ,x_0 \right)
        \end{split}
    \end{equation*}
for all $l\geq 2$ and $(v_0,v_1,\ldots,v_l)\in \mathcal{P}(\mathcal{G})$. Note that if we prove $\alpha$-Hölder continuity of $r$ restricted to $\Sigma^{term}$, we immediately obtain that it extends to an $\alpha$-Hölder continuous function on $\Sigma$. By Proposition \ref{proposition: displacements and the Cayley graph}, it suffices to show that if $v=(v_0,v_1,v_2,\ldots,v_l,0,0,\ldots)\in \Sigma^{term}$, and $v'=(v'_0,v'_1,v'_2,\ldots,v'_{l'},0,0,\ldots)$ $\in \Sigma^{term}$ satisfy $v_0=v'_0,\cdots,v_p=v'_p\neq 0$ for some $p\in \mathbb{N}$, then we have that the associated Gromov product satisfies
\[\left(\overline{\lambda(v_0,v_1)\cdots \lambda(v_{l-1},v_l)},\overline{\lambda(v'_0,v'_1)\cdots \lambda(v'_{l'-1},v'_{l'})}\right)_e\geq p,\]
which follows immediately from Lemma \ref{lemma: all words are geodesic} and noting that $2(g,h)_e=\abs{g}_{{\Gamma_0}}+\abs{h}_{{\Gamma_0}}-\abs{gh^{-1}}_{{\Gamma_0}}$ for all $g,h\in \Gamma$.
\end{proof}
\section{Thermodynamic Formalism for the Redfern Coding}
The goal of this section is to show that the quasiconvexity of some subgroup $H$ ensures the transfer operator associated with a counting problem for right-cosets in $H\backslash\Gamma$ is sufficiently regular. 

\subsection{Transfer operators and Spectral Decomposition.}
We explain in this subsection how the point spectrum of a transfer operator associated to a subshift of finite type decomposes into point spectra associated with transfer operators over irreducible subshifts. Let us first recall some terms from the spectral theory of operators on Banach spaces. Let $B$ be a Banach space and let $T:B\to B$ be a bounded linear operator. Denote the identity operator by $I$. We define the \textbf{spectrum} $\mathrm{spec}(T)$ to be the set of all $\lambda$ for which
\[(T-\lambda I)\]
does not have an inverse. We follow Browder \cite{Browder1961} in defining the \textbf{essential spectrum} $\mathrm{ess}(T)$ to be the set of $\lambda \in \mathrm{spec}(T)$ for which 
\begin{itemize}
    \item $\lambda$ is a limit point of $\mathrm{spec}(T)$, or
    \item $(T- \lambda I)B$ is not closed in $B$, or
    \item $\bigcup_{r=1}^\infty \mathrm{ker}(T-\lambda I)^r$ is infinite dimensional. 
\end{itemize}
We define the \textbf{spectral radius} of $T$ to be $\rho(T):= \sup\{|\lambda|:\lambda \in \mathrm{spec}(T)\}$ and the \textbf{essential spectral radius} of $T$ to be $\rho_e(T):= \sup\{|\lambda|:\lambda \in \mathrm{ess}(T)\}$. 

In what follows, let $(\mathcal{G},*,\lambda)$ be a Redfern coding associated to $\Gamma$ and some quasiconvex subgroup $H$ of $\Gamma$. Let $\Sigma$ be the associated subshift of finite type as defined in Subsection \ref{subsection:Subshifts of finite type}. Let $C^{\alpha}(\Sigma)$ denote the Banach space of $\alpha$-Hölder continuous functions on $\Sigma$. This Banach space is equipped with the natural norm 
\[ \left\| f\right\|_{C^{\alpha}} := \| f\|_\infty + |f|_\alpha,\]
where $\|\cdot\|_\infty$ is the standard supremum norm and 
\[|f|_\alpha = \sup\left\{\frac{|f(v)-f(v')|}{d_{\Sigma}(v,v')^{\alpha}}: v,v'\in \Sigma, v\neq v'\right\}.\]

Since the function $r$ in Proposition \ref{proposition: existence of roof function} is $\alpha$-Hölder continuous, the following is well-defined.
\begin{definition}\label{definition: transfer operator}
We define for $s\in \mathbb{C}$ the transfer operator
\[\mathcal{L}_s: C^{\alpha}(\Sigma)\to C^{\alpha}(\Sigma): f\mapsto \mathcal{L}_s(f),\]
where
\[\mathcal{L}_s(f)(x) = \sum_{\substack{\sigma y = x\\y\neq (0,0,\ldots)}} e^{-sr(y)}f(y). \]
\end{definition}
\begin{remark}
    The condition $y\neq (0,0,\ldots)$ changes the value of $\mathcal{L}_s(f)(x)$ only at $x=(0,0,\ldots)$. The only effect on the spectrum is to exclude the eigenvalue $1$ associated with the characteristic function on the singleton $\{(0,0,\ldots)\}$ (cf. \cite{Pollicott1998ComparisonTA}).
\end{remark}
Let us assume that we have already assigned an order to the vertex set $V\cup \{0\}$ so that the transition matrix $A$ is of the form \eqref{equation: Bs}. We define an auxiliary shift space $\widehat{\Sigma}\subset \Sigma$ consisting of all sequences $(v_1,v_2,\ldots)\in \left(V\cup \{0\}\right)^\mathbb{N}$ such that $\widehat{A}(v_i,v_{i+1})=1$, where $\widehat{A}$ is the transition matrix
\[\begin{pmatrix}
        B_{11} & 0 & \cdots &0\\
        0 & B_{22} & \cdots &0\\
        \vdots & \vdots & \ddots & \vdots\\
        0 & 0 &\cdots & B_{kk}
     \end{pmatrix}\]
and define $\widehat{\mathcal{L}}_s:C^\alpha(\widehat{\Sigma})\to  C^\alpha(\widehat{\Sigma})$ by
\[ \widehat{\mathcal{L}}_s(f)(x)= \sum_{\substack{\sigma y = x\\ y\in \widehat{\Sigma}}} e^{-sr(y)}f(y). \] 
The spectra of $\mathcal{L}_s$ and $\widehat{\mathcal{L}}_s$ are related by the following lemma
\begin{lemma}[Lemma 2 in \cite{Pollicott1998ComparisonTA}]\label{lemma: quasicompactness}
    Let $s = a+it$. The operators $\mathcal{L}_s$ and $\hat{\mathcal{L}}_s$ are both quasi-compact, with their essential spectral radii satisfying
    \[ \rho_e(\mathcal{L}_s)\leq 2^{-\alpha}\rho(\mathcal{L}_a) =2^{-\alpha}\rho(\widehat{\mathcal{L}}_a)\geq  \rho_e(\widehat{\mathcal{L}}_s).\]
    Furthermore the isolated eigenvalues of both operators coincide in algebraic multiplicity.
\end{lemma}
\begin{remark}\label{remark: geometric multiplicities may not coincide}
    The proof of Lemma 2 in \cite{Pollicott1998ComparisonTA} does not imply that the geometric multiplicities coincide. 
\end{remark}
The advantage of considering a transfer operator on $C^\alpha(\widehat{\Sigma})$ is that we obtain a natural decomposition $C^\alpha(\widehat{\Sigma})= C^\alpha(\Sigma_1)\times \cdots \times C^\alpha(\Sigma_k)$, where we recall that $\Sigma_l$ is the shift space associated to the irreducible component $B_{ll}$ for all $l$. If we define $\hat{\mathcal{L}}_{s,l}$ to be the operator $ \hat{\mathcal{L}}_s$ restricted to $C^\alpha(\Sigma_l)$, then we obtain that 
\begin{equation}\label{equation: decomposing the spectrum}
    \mathrm{spec}\left(\hat{\mathcal{L}}_s \right )=\bigcup_{l=1}^k \mathrm{spec}\left(\hat{\mathcal{L}}_{s,l} \right ).
\end{equation}
Irreducibility of $\Sigma_l$ implies the following. 
\begin{proposition}[Ruelle-Perron-Frobenius theorem]\label{Proposition: RPF}
Let $a\in\mathbb{R}$. For each $l\in \{1,\ldots,k\}$, the transfer operator $\hat{\mathcal{L}}_{a,l}$ has a simple, maximal positive eigenvalue $\lambda_{max}^{(l)}(a)$. Furthermore, there exists an $N_l\in \mathbb{N}$ such that the only eigenvalues on the circle $\{e^{i\phi}\lambda_{max}^{(l)}(a): \phi\in \mathbb{R}\}$ are simple eigenvalues with value $e^{ i2\pi p/N_l}\lambda_{max}^{(l)}(a)$ for $p\in \{0,1,\ldots, N_l-1\}$, with the rest of the spectrum contained in a disk of smaller radius. 
\end{proposition}
The proof is essentially contained in Chapters 1 and 2 of \cite{AST_1990__187-188__1_0}. While it is formulated for transfer operators over aperiodic shifts, we note that the proof of the existence of a simple, maximal positive eigenvalue requires only transitivity, and the other eigenfunctions with an eigenvalue of maximal modulus may be explicitly calculated after appropriately partitioning $\Sigma_l$.

We remark that $N_l$ is the greatest common divisor of the lengths of the closed loops in the subgraph of $\tilde{\mathcal{G}}$ spanned by the index set $V_l$. 

The quantity $P_l(-ar):= \log \lambda_{max}^{(l)}(a)$ is called the \textbf{pressure} of the function $-ar$ restricted to $\Sigma_l$. The following follows from well-known results for subshifts of finite type, see Chapters 3 and 4 in \cite{AST_1990__187-188__1_0}.
\begin{lemma}\label{lemma: pressure is zero at some point}
    The functions $\mathbb{R}\ni t\mapsto P_l(-tr)$ are monotonically decreasing and real analytic, and extend to a holomorphic function on a complex neighbourhood of $\mathbb{R}$ in $\mathbb{C}$. Furthermore there exists a real number $a_l\in \mathbb{R}$ such that 
\begin{itemize}
    \item $P_l(-a_lr) = 0$
    \item $\rstr{\frac{d}{dt} P_l(-tr)}{t=a_l}<0$.
\end{itemize}
\end{lemma}
We end this subsection with a well-known lemma that will allow us to analyse the behaviour of the operators $\widehat{\mathcal{L}}_{s,l}$ on vertical lines. Denote for any observable $f:\Sigma\to\mathbb{C}$ its Birkhoff sum by $f^l(x):=\sum_{j=0}^{l-1}f(\sigma^j(x))$.
\begin{lemma}\label{lemma: lattice condition}
    For all $s=a+it\in \mathbb{C}$, we have that $\rho(\widehat{\mathcal{L}}_{s,l})\leq \rho(\widehat{\mathcal{L}}_{\Re(s),l})$. Furthermore, there exists some $t\in \mathbb{R}-\{0\}$ for which $e^{P_l(-ar)}\in \mathrm{spec}(\widehat{\mathcal{L}}_{a+it,l})$ if and only if the restriction $\rstr{r^l}{\Sigma_l}$ of $r$ to $\Sigma_l$ satisfies
    \[\left\{\rstr{r^l}{\Sigma_l}(x):l\in\mathbb{N} \text{ and } \sigma^l(x)=x\right\}\subset b\mathbb{Z},\]
    where $b=\frac{2\pi}{|t|N_l}$ and $t$ is chosen to be of minimal magnitude.
\end{lemma}
\begin{proof}
    This follows from the well-known fact that $e^{P_l(-ar)}\in \mathrm{spec}(\widehat{\mathcal{L}}_{a+it,l})$ if and only if there is some Hölder continuous function\footnote{Possibly with a different Hölder exponent} $u:\Sigma\to \mathbb{R}$, a continuous function $\Psi:\Sigma\to \mathbb{Z}$ such that $ tr= u\circ\sigma-u+\frac{2\pi}{N_l} \Psi$, which follows easily from the theorem in the aperiodic case (Theorem 4.5 in \cite{AST_1990__187-188__1_0}).
\end{proof}

\begin{definition}\label{definition: maximal component}
   Define $a:=\max_{l\in \{1,\ldots,k\}} a_l$ as in Lemma \ref{lemma: pressure is zero at some point}. We say $\Sigma_l$ is a \textbf{maximal component} if $P_l(-ar)=0$.
\end{definition}
\subsection{Poincaré series}
We now consider a counting problem for orbit points restricted to those generated by group elements corresponding to words in $\mathfrak{L}(\mathcal{G})$ under the condition that we fix the first $l$ entries of $x\in \mathfrak{L}(\mathcal{G})$. 
 For a path $x=(x_0,x_1,x_2,\ldots ,x_p)\in \mathcal{P}(\mathcal{G})$ and $l,m\in \{0,1,\ldots, k\}$ with $l\leq m$, we define $x_l^m:=(x_l,x_{l+1},\ldots, x_m)$. Furthermore, we denote the \textbf{length} of $x$ by $\abs{x}=p+1$ and the \textbf{displacement} by \begin{equation}\label{equation: definition: displacement}
    L_x:= d\left(\overline{\lambda(x)}\cdot x_0,  x_0\right).
\end{equation}
Consider for $u\in \mathfrak{L}(\mathcal{G})$ and $m\geq \abs{u}$ the set
\[J_m(u):=\left\{ x\in \mathfrak{L}(\mathcal{G}): \abs{x}=m, x_{0}^{\abs{u}-1} = u\right\}.\]
We have the following proposition for counting orbits restricted to the set $J_m(u)$.  
\begin{proposition}\label{proposition: counting representatives of cosets}
Assume the geodesic flow on $X/\Gamma$ is mixing. Assume we have a Redfern coding $(\mathcal{G},*,\lambda)$ associated to a quasiconvex subgroup $H$ such that $\partial H\neq \partial\Gamma$. Let $u\in \mathfrak{L}(\mathcal{G})$. There exists a constant $C_{u}\geq 0$ such that 
\[N^{u}(T) :=\#\left\{x\in \mathfrak{L}(\mathcal{G}): x_{0}^{{\abs{u}}-1} = u, L_x \leq T\right\}\]
satisfies
\[\lim_{T\to \infty} e^{-\delta T}N^{u}(T)=C_u\]
as $T\to \infty$, where $\delta$ is the exponential growth rate in Roblin's orbital counting theorem, i.e. Proposition \ref{theorem: Margulis} in the $\CAT$-setting. 
\end{proposition}

The proof of this proposition requires several steps. We first show no two maximal components are connected in the component graph we defined in Definition  \ref{definition: components graph}, which will ensure semi-simplicity of the eigenvalue $1$ of the operator $\mathcal{L}_a$. This will follow from an initial coarse counting estimate provided in Lemma \ref{lemma: Growth rate of exponent of Transfer operator} below. Let us make explicit the connection between the transfer operator and the counting problem. 
Consider the Laplace transform of the distributional derivative of $N^{u}$, i.e.
\begin{equation}\label{equation: Laplace transform}
    \eta_{u}(s)  := \int_0^{+\infty} e^{-sT} dN^{u}(T) =\sum_{m=\abs{u}}^{+\infty} \sum_{b\in J_m(u)} e^{-sL_b}.
\end{equation}
We obtain by the definition of $r$ in Proposition \ref{proposition: existence of roof function} that 
\[\sum_{b\in J_m(u)} e^{-sL_b} = \sum_{b\in J_m(u)} e^{-sr^m(b\mathbf{0})},\]
where $b\mathbf{0} = (b_1,\ldots ,b_m,0,0,\ldots)$.
Define the function $\chi_{[u]} \in C^\alpha(\Sigma)$ to be the characteristic function on the cylinder set $[u]:= \{(x_0,x_1,\ldots)\in \Sigma: u=(x_0,x_1,\ldots, x_{\abs{u}-1})\}$. We use the definition of the transfer operator to obtain that when $\Re(s)$ is sufficiently large,
\begin{equation}\label{Equation: simplifying the Poincaré series}
    \eta_{u}(s) = \sum_{m=\abs{u}}^{+\infty} \mathcal{L}_s^{m}\chi_{[u]}(\mathbf{0})= \sum_{m=0}^{+\infty} \mathcal{L}_s^{m}\chi_{[u]}(\mathbf{0}),
\end{equation}
where we use the fact that $\chi_{[u]}(b\mathbf{0})=0$ when $\abs{b}< \abs{u}$. It is straightforward to see that the assignment $s\mapsto \mathcal{L}_s$ is holomorphic as an operator-valued function and that it follows from Lemma \ref{lemma: quasicompactness} and the classical analytical perturbation theory of operators that the resolvent
\[ s\mapsto (1-\mathcal{L}_s)^{-1}\]
is holomorphic as a $C^\alpha(\Sigma)$-function on the half plane $\Re(s)>a$ and that there is some $\varepsilon>0$ for which it is meromorphic on $\Re(s)>a-\epsilon$, see e.g. Theorem 1.9 in Chapter VII of \cite{Kato} along with the supplementary notes for the chapter, whence 
\begin{equation}\label{equation: eta_uv using resolvent}
    \eta_{u}= (1-\mathcal{L}_s)^{-1}\chi_{[u]}(\mathbf{0}) 
\end{equation}
also extends to a meromorphic function on $\Re{s}>a-\epsilon$. 

In order to obtain an asymptotic for the counting problem, we shall show that the only pole of $\eta_{u}$ on the line $s= a+it$ is a simple pole at $t=0$, which will follow from an initial rough estimate for the counting problem.
We shall show that if $\partial H\neq \partial\Gamma$, the representatives of the cosets have, up to a multiplicative constant, the same growth as in the unrestricted case.
\begin{lemma}\label{lemma: exponential growth rate}
    Assume that $\partial H\neq \partial\Gamma$ and let $\delta$ be the constant in Roblin's orbital counting theorem \cite{Roblin03}.  For $t>0$, let $N_H(t)$ be the number of paths $w\in \mathfrak{L}(\mathcal{G})$ with $L_w\leq t $. Then $ N_H(t) \geq  \varepsilon e^{t\delta}$ for some $\varepsilon>0$.  
\end{lemma}
\begin{proof}
        We adapt the argument of Proposition 5 in \cite{Parkkonen2015}. The orbit map $h\mapsto h\cdot x_0$ induces a bi-Hölder bijection from $\partial \Gamma$ to the limit set $ \Lambda(\Gamma)$ of $\Gamma$ on $X$, see \cite{Ghys90}. In particular, we have that $\Lambda(H)\neq \Lambda(\Gamma)$. Note that $H$ acts properly discontinuously \footnote{Here we're defining an action to be properly discontinuous if for any compact set $K$, the set $\{g\in H: g\cdot K \cap K \neq \emptyset\}$ is finite} on $X\cup (\Lambda(\Gamma)-\Lambda( H))$, otherwise $H$ would have an accumulation point in $X\cup \Lambda(\Gamma)$ not in $\Lambda( H)$, which is a contradiction. Let $p\in \Lambda(\Gamma)\backslash \Lambda(H)$. Since $\Lambda(H)$ is a compact set and the topology on $X\cup \Lambda( \Gamma)$ is metrizable, there exists an open neighbourhood $U$ of $p$ in $X \cup \Lambda(\Gamma)$ such that its closure $\overline{U}$ still satisfies $\overline{U}\cap \Lambda( H) = \emptyset$. Then $\overline{U}$, and hence $U$, intersects only finitely many of its $H$-translates. 
        
        If $V\subset \Gamma$ denotes the pre-image of $U$ under the orbit map $h\mapsto h\cdot x_0$ and $(\cdot,\cdot)_\cdot$ denotes the Gromov product with respect to the word metric, then there is some $M$ such that $(a,h)_e\leq M$ for all $a\in V $ and $h\in H$, otherwise a compactness argument shows that the closure $\overline{V}$ intersects $\partial H$. If $h\in H$ such that $h^{-1}a$ has minimal word length, we have by the relations
        \[(a,h)_e=\frac{1}{2}\left(\abs{a} +\abs{h}-\abs{h^{-1}a}\right) \ \text{ and } \abs{a} -\abs{h^{-1}a}\geq 0,\] the fact that $\abs{h}\leq 2M$
        and quasiconvexity that $\langle a\cdot x_0,h\cdot x_0\rangle_{x_0}\leq M'$ uniformly, where $M'>0$ is some constant and $\langle\cdot,\cdot\rangle_\cdot$ is the Gromov product on $X$. We thus have by definition of the Gromov product and finiteness of $\abs{h}$ that $\abs{d(a\cdot x_0, x_0) -d(a\cdot x_0,h\cdot x_0)}$ is uniformly bounded, i.e. for any orbit point $a\cdot x_0\in U$, we have that the difference between the displacement $d(a\cdot x_0,x_0)$ and $d(b\cdot x_0, x_0)$ is uniformly bounded for any $b$ with $Hb = Ha$ of minimal word length.         
        
        By this and the fact that $\#\{h\in H: h\cdot U\cap U \neq \emptyset\}$ is finite, it suffices to show that the cardinality of $\{g\in \Gamma: d(g\cdot x_0, x_0)\leq R\}\cap U$ can be bounded below by $C^{-1}\#\{g\in \Gamma: d(g\cdot x_0, x_0)\leq R\}$, for $C$ large enough, which is a fairly standard argument. For example, use minimality of the action on $\partial\Gamma$ and compactness of $\Gamma \cup \partial \Gamma$ to argue that $\Gamma \cup \partial \Gamma=g_1 V \cup \cdots\cup g_m V$ for some $g_1,\ldots,g_m\in \Gamma$, whence $g_1\cdot U \cup \cdots \cup g_m\cdot U$ contains all the orbit points. 
\end{proof}
We shall use this to argue that the resolvent $(1-\mathcal{L}_s)^{-1}$ has a simple pole at $s=\delta$. Let us first relate $\delta$ with the pressure functions associated to the maximal components.
\begin{lemma}
The quantity $a:= \sup_{l} a_l$, where $a_l$ is defined as in Lemma \ref{lemma: pressure is zero at some point}, is the \textbf{exponent of convergence} of the Poincaré series
\[\eta_\mathcal{G}(s)=\int_0^{+\infty} e^{-sT} dN_H(T) = \sum_{w\in \mathfrak{L}(\mathcal{G})}e^{-sL_w},\]
i.e. the infimum over all $s\in (0,\infty)$ for which the above converges.
\end{lemma}
\begin{proof}
    By a similar argument as the one we employed for obtaining a formula for $\eta_{u}$, we see that
\[\eta_\mathcal{G}(s)=\sum_{p=0}^{+\infty} \mathcal{L}_s^{p}(\chi_{[*]})(\mathbf{0})= (1-\mathcal{L}_s)^{-1}(\chi_{[*]})(0),\]
where $\chi_{[*]}$ is the indicator function over all sequences in $\Sigma$ which start at $*$. Since the spectral radius of $\mathcal{L}_{s}$ is less than unity when $\Re(s)>a$, it suffices to prove that $\lim_{M\to \infty} \sum_{p=1}^{M} \mathcal{L}_a^{p}(\chi_{[*]})(\mathbf{0})=\infty$. This follows straightforwardly from Lemma \ref{lemma: Growth rate of exponent of Transfer operator} below.
\end{proof}

We shall later show in Lemma \ref{Lemma: maximal components are not connected} that the above estimate forces $\eta_{\mathcal{G}}$ to have a simple pole at $s=\delta$. To also show that this implies no two maximal components in the component graph are connected by an edge, we need the following result which gives a lower bound for the growth of $ \mathcal{L}_s^{p}(\chi_{[*]})(\mathbf{0})$ for real $s$ as $p\to \infty$.
\begin{lemma}\label{lemma: Growth rate of exponent of Transfer operator}
    Let $m$ be the largest number such that there exist maximal components $V_{l_1},\ldots,V_{l_m}$ and a path $(v_0,v_1,v_2,\ldots)$ of connected vertices which enter into each of these components. Then there exist constants $C,\varepsilon >0$ such that for all $t\in [a-\epsilon,a+\epsilon]$ and $p\geq 1$,
            \[\mathcal{L}_t^p(\chi_{[*]})(\mathbf{0})\geq Cp^{m-1}e^{ps(t)},\] 
   where \[s(t)=\min\left\{{P_{j}(-tr)}: \Sigma_j \text{ is a maximal component}\right\} .\]
\end{lemma}
\begin{proof}
This is essentially contained in \cite{Gouezel14}, Lemma 3.7, where only one Hölder continuous potential is considered. But it is straightforward to generalise the proof to a sufficiently small real-valued perturbation of a real-valued roof function.     
\end{proof}

By Lemma \ref{lemma: exponential growth rate} we have that $N_H(t)\geq \varepsilon e^{\delta t}$ as $t\to \infty$. On the other hand, it is clear that 
\[N_H(T)\leq N_{\{e\}}(T):= \#\{h\in \Gamma: d(h\cdot x_0,x_0)\leq T\}\]
for all $T$, hence we have by Roblin's orbital counting result (cf. Chapitre 4 in \cite{Roblin03}) that for $C$ sufficiently large, there is some $t_0>0$ such that $C^{-1} e^{t\delta} \leq N_H(t) \leq Ce^{t\delta}$ for all $t\geq t_0$. We show that this implies the following.
\begin{lemma}\label{Lemma: maximal components are not connected}
 If $m$ is the number defined in Lemma \ref{lemma: Growth rate of exponent of Transfer operator}, then $m=1$.
\end{lemma}
\begin{proof}
By Lemma \ref{lemma: Growth rate of exponent of Transfer operator}, we have for all $t\in (a,a+\varepsilon)$ that \[\eta_{\mathcal{G}}(t)=(1-\mathcal{L}_t)^{-1}(\chi_{[*]})(\mathbf{0})\geq C\mathrm{Li}_{1-m}(e^{s(t)}),\] where we define the \textbf{polylogarithm} for $s\in \mathbb{C}$ and $\abs{z}<1$ by the power series $\mathrm{Li}_{s}(z)=\sum_{p=1}^\infty z^p/p^s$. It follows from standard facts about the (meromorphic continuation of the) polylogarithm that $\mathrm{Li}_{1-m}(z)$ has a pole of order $m$ at $z=1$. By Lemma \ref{lemma: pressure is zero at some point}, the pressure functions $P_j(-tr)$ associated to the maximal components are analytic and have strictly negative derivative at $t=a$. Therefore, there exist constants $c_1,c_2>0$ such that $-c_2(t-a)\leq s(t)\leq -c_1(t-a)$ holds for $t$ sufficiently close to $a$, whence there exists some $\kappa>0$ such that \[\mathrm{Li}_{1-m}(e^{s(t)})\geq \frac{\kappa}{(t-a)^m}\] for $t$ sufficiently close to $a$. Using the fact that $\eta_{\mathcal{G}}(s)$ is meromorphic near $s=a$, we have that $\eta_\mathcal{G}(s)$ has a pole of order at least $m$ at $s=a$. By the discussion before the statement of this lemma, we have that $C^{-1} e^{t\delta} \leq N_H(t) \leq Ce^{t\delta}$ for $C$ large enough and $t\geq t_0>0$, which also implies $a=\delta$. We see by taking the Laplace transform that for $\epsilon>0$ we have that 
\begin{equation}\label{equation: inequality for Laplace transform}
\frac{C^{-1}}{\epsilon}-K\leq \int_0^\infty e^{-(\delta+\epsilon) t}N_H(t) dt\leq \frac{C}{\epsilon}+K,
\end{equation}
where $K$ is some additive constant related to the integrand in the interval $[0,t_0]$. By partial integration, we have that \[\eta_\mathcal{G}(\delta+\epsilon)=(\delta+\epsilon)\int_0^\infty e^{-(\delta+\epsilon) t}N_H(t)dt.\] 
By \eqref{equation: inequality for Laplace transform}, we obtain for $D>0$ sufficiently large and $\epsilon>0$ sufficiently small, that 
\[\frac{D^{-1}}{\epsilon} \leq \eta_\mathcal{G}(\delta+\epsilon)\leq  \frac{D}{\epsilon}.\]
Recalling that $s\mapsto \eta_\mathcal{G}(s)$ is meromorphic near $s=\delta$ we conclude that it has a simple pole at $s=\delta$, which implies that $m=1$ and hence that there exists no path from one maximal component to another. 
\end{proof}

Lemma \ref{Lemma: maximal components are not connected} implies (cf. the proof of Lemma 4.4 in \cite{CalegariFujiwara10}) that the algebraic multiplicity and the geometric multiplicity of the eigenvalue $1$ of $\mathcal{L}_a$ are the same. In fact the next result shows that this also holds in a neighbourhood of $s=a$. Consider the projection $R(a):C^{\alpha}(\Sigma)\to C^{\alpha}(\Sigma)$ onto the eigenspace of $\mathcal{L}_a$ of eigenvalue $1$. By perturbation theory (see \cite{Kato}) we have that $R(a)$ extends to an analytic function $s\mapsto R(s)$ of constant rank on some neighbourhood of $a$ such that the spectrum of $(1-R(s))\mathcal{L}_s$ remains bounded away from $1$. 
\begin{proposition}[Theorem 3.8 and Proposition 3.10 in \cite{Gouezel14}]\label{proposition: maximal eigenvalue splits}
Let $\mathcal{M}\subset\{1,\ldots,k\}$ such that if $j\in \mathcal{M}$, then $\Sigma_j$ is a maximal component. There exists a small complex neighbourhood $U$ of $a$ such that for all $s\in U$, there exists for each maximal component $V_{j}$ an eigenfunction $h^{(j)}_s\in C^\alpha(\Sigma)$ with eigenvalue $e^{P_{j}(-sr)}$ and a corresponding eigenmeasure  $\mu_s^{(j)}\in \left(C^\alpha(\Sigma)\right)^*$ varying holomorphically with $s$ such that
 \[R(s) = \sum_{j\in \mathcal{M}} R_j(s)\text{, where } R_j(s)f = h^{(j)}_s\int f\,d\mu_s^{(j)}.\]
 Furthermore, the eigenfunction $h^{(j)}_a$ is supported and strictly positive on the set of sequences in $\Sigma$ starting at a vertex $v_0$ with $v_0\in V_j$ or $v_0\in V_l \succ V_j$ for some $l$ (see Definition \ref{definition: components graph}). The eigenmeasure $\mu_a^{(j)}$ is supported on the set of all non-terminating sequences in $(v_p)_p\in\Sigma$ such that $v_p\in \Sigma_{j}$ for infinitely many $p$.
\end{proposition}
\begin{remark}
    In fact, Theorem 3.8 and Proposition 3.10 in \cite{Gouezel14} are formulated using functions and measures associated to the cyclic decomposition of the maximal components. In particular, these results show that near $s=a$, there exist functions $h^{(j)}_{s,p}$ and measures $\mu^{(j)}_{s,p}$ for $p\in\{0,\ldots, N_j-1\}$ such that \[\mathcal{L}_sh^{(j)}_{s,p} =e^{P_{j}(-sr)}h^{(j)}_{s,p+1}\text{ and }\mathcal{L}_s^*\mu^{(j)}_{s,p} =e^{P_{j}(-sr)}\mu^{(j)}_{s,p-1 },\]
    where $p+1$ and $p-1$ are taken modulo $N_j$. 
    The decomposition in the statement of Proposition \ref{proposition: maximal eigenvalue splits} therefore follows for $h^{(j)}_s=\sum_{p=0}^{N_j-1} h^{(j)}_{s,p}$ and $\mu_s^{(j)}=\left(\sum_{p=0}^{N_j-1}\mu^{(j)}_{s,p}\right)/N_j$.

\end{remark}
By applying the splitting $\mathcal{L}_s = \mathcal{L}_s R(s) + \mathcal{L}_s (1-R(s))$, we see that
 \[\eta_{u}(s)=  \sum_{j\in \mathcal{M}}\frac{h^{(j)}_s(\mathbf{0})\mu_s^{(j)}([u])}{1-e^{P_{j}(-sr)}}+Q(s),\]
where $Q(s)$ is some holomorphic function in $s$. Letting $s\to a$ and using the results on the support of $h^{(j)}_s,\mu_s^{(j)}$ of Proposition \ref{proposition: maximal eigenvalue splits} we obtain an explicit expression for the residue of the pole of $\eta_u$ near $s=a$. 
\begin{proposition}\label{proposition: form of Cuv}
There exists a holomorphic function $U\ni s\mapsto Q'(s)$ such that
 \[
\eta_u(s)=\frac{C'_u}{s-a}+Q'(s),
\]
where
\[
C'_u=
\sum_{j\in\mathcal M}
\frac{h_a^{(j)}(\mathbf 0)\mu_a^{(j)}([u])}
{-\left.\frac{d}{dt}P_j(-tr)\right|_{t=a}},
\]
and the contribution of each $j\in \mathcal{M}$ to the above sum is nonzero if and only if there exists a path $w\in \mathfrak{L}(\mathcal{G})$ with $w_{0}^{{\abs{u}-1}}=u$, such that $w$ contains a vertex belonging to $\Sigma_j$. 
 \end{proposition}
\begin{proof}
    It is clear from the properties of $h_a^{(j)}$ that the inequality \[\frac{h_a^{(j)}(\mathbf{0})}{-\rstr{\frac{d}{dt}P_j(-tr)}{t=a}}>0\]
    holds. The fact that $\mu_a^{(j)}([u])>0$ when there exists a path $w$ with $w_{0}^{{\abs{u}-1}}=u$ follows from the explicit construction of the eigenmeasure in \cite{Gouezel14}.
\end{proof}

Let us recall that $\eta_u$ is the Laplace transform of the distributional derivative of the counting function $N^u$ defined in the statement of Proposition \ref{proposition: counting representatives of cosets} and that $\eta_u$ is holomorphic on $\Re(s)>a$. Recalling also that we showed $a=\delta$ in the proof of Lemma \ref{Lemma: maximal components are not connected}, we shall prove Proposition \ref{proposition: counting representatives of cosets} by using the Wiener-Ikehara Tauberian theorem. In order to apply this, we need to show that the map $s\mapsto \eta_u(s)-\frac{C_{u}'}{s-a}$ has a continuous extension to $\Re(s)\geq a$. Since $\eta_u$ is meromorphic in $\Re(s)> a-\epsilon$ for some $\epsilon>0$, it suffices to prove that $\eta_{u}$ has no other poles on the line $\Re(s)=a$ except at $s=a$. To show this, we use Lemma \ref{lemma: lattice condition} to argue that the roof function restricted to a maximal component being cohomologous to a lattice-valued function implies that \enquote{too many} closed geodesics have length in $b\mathbb{Z}$ for some $b$. This in turn relies on Lemma \ref{lemma: period growth is exponential}, which is a standard counting result for orbits on shift spaces, along with the more technical Lemma \ref{lemma: finite-to-one}, which shows these closed orbits indeed correspond to closed geodesics. 
\begin{lemma}\label{lemma: period growth is exponential}
    Let $(\Sigma_D,\sigma)$ be a shift space associated to an irreducible matrix. Let $f\in C^\alpha(\Sigma_D)$ be real-valued and consider the family of transfer operators defined by the expression $\mathcal{I}_s(w)(x):=\sum_{\sigma(y)=x} e^{-sf(y)}w(y)$. Let $\delta_D\in\mathbb{R}$ such that $\rho(\mathcal{I}_{\delta_D})=1$. Let $\pi_D(T)$ be the number of distinct periodic orbits $\{\sigma(y),\sigma^2(y),\ldots, \sigma^k(y)=y\}$ satisfying $f^k(y)\leq T$. Then 
    \begin{equation}\label{equation: exponentially many periodic orbits}
        \pi_D(T)\sim C(T)e^{\delta_D T}/{T},  
    \end{equation}
    as $T\to\infty$, where $C(T)$ is $b$-periodic for some $b>0$. Furthermore $C(T)$ is constant if and only if $1\notin \mathrm{spec}(\mathcal{I}_{\delta_D+it})$ for all $t\neq 0$.
\end{lemma}
\begin{proof}
This is a well-known result, see e.g. the proof of Theorem 2 in \cite{ParryPollicott83}.
\end{proof}

Recall that an element $g\in \Gamma$ is hyperbolic if it fixes two distinct points $\xi,\eta$ in $\partial X$. We define the \textbf{axis} $\mathrm{ax}(g)$ of $g$ to be the geodesic with endpoints $\xi$ and $\eta$. On this axis, $g$ acts by translation, and we endow $\mathrm{ax}(g)$ with an orientation towards the attracting endpoint with respect to this translation. We define as in \cite{Roblin03} a \textbf{closed geodesic} to be the projection of an axis of a hyperbolic element to $X/\Gamma$. It follows from this definition that there is a one-to-one correspondence between oriented closed geodesics\footnote{we do not require these geodesics to be primitive} and conjugacy classes of hyperbolic elements. Furthermore, mixing of the geodesic flow implies the following counting result for periodic orbits. Let $\pi_{X/\Gamma}(t)$ be the number of closed geodesics of lengths less than or equal to $t$. Then there exists some constant $C>0$ for which 
 \begin{equation}\label{equation: Roblins orbit counting}
     \pi_{X/\Gamma}(t) \sim C\frac{e^{\delta t}}{t},
 \end{equation}
see Chapitre 5 in \cite{Roblin03}. 
\begin{lemma}\label{lemma: finite-to-one}
    There exists an $N\in\mathbb{N}$ such that for any conjugacy class $\mathfrak{C}$ of a hyperbolic element in $\Gamma$ and any $n\in \mathbb{N}$, there are at most $N$ distinct closed orbits 
    \begin{equation}\label{equation:periodic orbit}
        \{\sigma(v),\sigma^2(v),\ldots, \sigma^n(v)=v\}
    \end{equation}
    such that if we write $v=(v_1,\ldots,v_n,v_1,\ldots)$, then \[\overline{\lambda((v_1,v_2))\lambda((v_2,v_3))\cdots \lambda((v_n,v_1))}\in \mathfrak{C}.\]
\end{lemma}
\begin{proof}
    We first show that the axis associated to the group element represented by a periodic point in $\Sigma_D$ lies within a bounded distance of $x_0$.
    Let $\mathfrak{C}$ be as in the statement of the lemma and assume there is a corresponding orbit $\{\sigma(y),\sigma^2(y),\ldots, \sigma^n(y)=y\}$ as in the statement of the lemma. Define  $g_j=\overline{\lambda(v_{j\text{ mod }n},v_{j+1 \text{ mod }n})}$. Define the path 
    \[\gamma: \mathbb{Z}\to \Gamma: p\mapsto \begin{cases}
        \prod_{j=1}^{p}g_j & \text{ if } p\geq 1\\
        e & \text{ if } p=0\\
        \left(\prod_{j=1}^{-p}g_j\right)^{-1} & \text{ if } p\leq -1.
    \end{cases}\]
    We note that this path is geodesic by Lemma \ref{lemma: all words are geodesic} and the fact that the set of generators is symmetric. We thus have that $\gamma$ is a geodesic connecting two points $\xi,\eta \in \partial \Gamma$ and passing through $e$. By quasi-isometry of the map $F:\Gamma \to X \ni h\mapsto h\cdot x_0$, we have that the map $F\circ \gamma$ is a quasi-geodesic, where we recall that a quasi-geodesic is a quasi-isometric map with domain $\mathbb{Z}$ or $\mathbb{R}$. If we denote $g=g_1g_2\cdots g_n$, we see that since quasi-geodesics which share the same endpoints remain within a uniformly bounded distance of each other (see Chapitre 5 in \cite{Ghys90}), and the axis $\mathrm{ax}(g)$ of $g$ has the same endpoints as the quasi-geodesic $F\circ \gamma$, we have that $\mathrm{ax}(g)$ passes within a bounded distance of $x_0=F\circ \gamma(0)$.
    
    We now prove that for a closed geodesic $\overline{\gamma}$ in $X/\Gamma$ with length $l(\overline{\gamma})$, or equivalently its associated conjugacy class $\mathfrak{C}$, there is a constant $C>0$ such that the number of periodic points in $\Sigma_D$ corresponding to elements in $\mathfrak{C}$ as above is less than $Cl(\overline{\gamma})$. It suffices to pick $g\in \mathfrak{C}$ and prove that for some $W>0$, \[\# \{hgh^{-1}:h\in \Gamma \text{ and } d(h\cdot\mathrm{ax}(g),x_0)\leq L\}\leq  Wl(\overline{\gamma})\]
    uniformly in $\overline{\gamma}$.
    If $\beta$ is a geodesic segment of $\mathrm{ax}(g)$ of length $l(\overline{\gamma})$, then 
    \[\# \{hgh^{-1}:h\in \Gamma \text{ and } d(h\cdot \mathrm{ax}(g),x_0)\leq L\} \leq \#\{h\in \Gamma: h\beta \cap \overline{B}_L(x_0)\neq \emptyset \},\]
    where $\overline{B}_L(x_0):=\{x\in X:d(x_0,x)\leq L\}$. It thus suffices to show that there is some $D>0$ such that any geodesic segment of length $1$ in $X$ has at most $D$ intersections with the set $\{g\cdot\overline{B}_L(x_0):g\in \Gamma\}$. But this follows immediately from the fact that the action of $\Gamma$ on $X$ is properly discontinuous and that $\overline{B}_L(x_0)$ is compact. 

    To finish the proof, it suffices to note that boundedness of the roof function implies that there is some $\epsilon$ such that a periodic orbit of the form \eqref{equation:periodic orbit} associated with $\mathfrak{C}$ satisfies $n>\epsilon l(\overline{\gamma})$. We may therefore take $N>C/\epsilon$.
\end{proof}
Suppose now that $1\in \mathrm{spec}(\mathcal{L}_{a+it})$ for some $t\neq 0$. By \eqref{equation: decomposing the spectrum}, we have that there is some irreducible maximal component $V_j$ such that $1\in \mathrm{spec} (\hat{\mathcal{L}}_{a+it,j})$. By Lemma \ref{lemma: lattice condition}, we have that \[\left\{\rstr{r^l}{\Sigma_j}(x):l\in\mathbb{N} \text{ and } \sigma^l(x)=x\right\}\subset b\mathbb{Z} \] for some $b>0$. It is immediately clear that every closed orbit of $\sigma$ in $\Sigma_j$ corresponds to a closed orbit in $\Sigma$, hence we obtain by Lemma \ref{lemma: period growth is exponential} and Lemma \ref{lemma: finite-to-one} that there is some $\epsilon>0$ for which the number of geodesics\footnote{These geodesics may not necessarily be prime, but that is irrelevant to the argument.} with lengths in $b\mathbb{Z}\cap[0,T]$ exceeds $\epsilon e^{\delta T}/T$ for infinitely many $T$ large enough. In particular this means there are infinitely many $k\in \mathbb{N}$ for which the number of geodesics with length $kb$ is greater than or equal to \[\epsilon\left( \frac{e^{\delta k b}}{k\delta b}-\frac{e^{\delta(k-1)b}}{(k-1)\delta b}\right).\] Hence there are infinitely many $k$ such that the number of geodesics of length $kb$ is greater than $\frac{\epsilon}{2}(e^{\delta b}-1)e^{\delta k b}/{\delta k b}$, which contradicts \eqref{equation: Roblins orbit counting}. 

We have thus shown that $1\notin \mathrm{spec}(\mathcal{L}_{a+it})$ for all  $t\neq 0$. Hence $\eta_u$ has no poles on the line $s=\delta$, and Proposition \ref{proposition: counting representatives of cosets} now follows straightforwardly. 
\begin{proof}[Proof of Proposition \ref{proposition: counting representatives of cosets}]
    By the previous discussion, we have that $\eta_{u}$ is of the form 
    \begin{equation}
        \eta_{u}(s)= \frac{C_{u}'}{s-a}+Q'(s),
    \end{equation}
    where $Q'(s)$ is holomorphic in some neighbourhood of $\Re(s)\geq a$. Suppose there exists at least one path from $u$ passing through a maximal component. Then we know by Proposition \ref{proposition: form of Cuv} that $C_{u}'>0$. Recalling the definition of $\eta_{u}$ as a Laplace transform, \eqref{equation: Laplace transform} the Wiener-Ikehara Tauberian theorem (see Theorem 5.1 in \cite{Korevaar}) states that 
\[e^{-\delta T}N^{u}(T)\sim \frac{C_{u}'}{a} \]
as $T\to +\infty$. The proposition follows after taking $C_u:= \frac{C_{u}'}{a}$.
\end{proof}

\section{Proving Theorem 1.1}
The purpose of this section is to introduce two more results which will allow us to derive Theorem \ref{Hypothesis:First Step} from Proposition \ref{proposition: counting representatives of cosets}. It is a standard fact that the centraliser $Z(g)$ of a group element $g\in \Gamma$ is a quasiconvex subgroup of $\Gamma$ (see \cite{Bridson99} p. 477), hence we have by Theorem \ref{theorem: Redfern} that there exists a Redfern coding $(\mathcal{G},*,\lambda)$ associated with $Z(g)$. We assume furthermore that $\mathrm{Cl}(g)$ is not finite, i.e. that $Z(g)$ is of infinite index in $\Gamma$, which is equivalent to the condition that $\partial Z(g)\neq \partial\Gamma$, see \S 5.1 in \cite{Champetier00}. Proposition \ref{Proposition: Estimating Lengths} uses Lemma \ref{definition: strong hyperbolicity} to compare the displacement $d\left(\overline{\lambda(x)}^{-1}g\overline{\lambda(x)}\cdot x_0,x_0\right)$ for some $x\in \mathfrak{L}(\mathcal{G})$ with the displacement $L_x$. Lemma \ref{Lemma: Convergence of the Sum} is then a summability result allowing us to deduce Theorem \ref{Hypothesis:First Step} from the aforementioned proposition. 

We start with the following proposition that compares the displacement of representatives of the right cosets of the centraliser with those of corresponding elements in the conjugacy class. For a word $u\in \mathfrak{L}(\mathcal{G})$, define $\mathcal{J}(u)=\bigcup_{k=\abs{u}}^\infty J_k(u)$. We have the following proposition. 
\begin{proposition}\label{Proposition: Estimating Lengths}
There exist constants $K>0$ and $0<\rho< 1$ such that for any $l\in\mathbb{N}$ and any word $u\in \mathfrak{L}(\mathcal{G})$  with $\abs{u}=l$, there exists a uniformly bounded real number $\tau(u)$ such that for any $x\in \mathcal{J}(u)$ we have that
\begin{equation}\label{Equation: Estimating Lengths}
    \left|d\left(\overline{\lambda(x)}^{-1}g\overline{\lambda(x)}\cdot x_0 ,x_0\right)-2L_x-\tau(u)\right|\leq K\rho^l,
\end{equation}
with \[\tau(u)=d\left(g\cdot x_0,x_0\right)-2\left\langle \overline{\lambda(u)}\cdot x_0,g\overline{\lambda(u)}\cdot x_0\right\rangle_{g\cdot x_0}-2\left\langle g\cdot x_0, \overline{\lambda(u)}\cdot x_0\right\rangle_{x_0}.\]
\end{proposition}
\begin{proof}
    We define some notational shorthand. Let $\hat{d}:\Gamma\times \Gamma \to [0,\infty)$ denote the pullback metric $\hat{d}(h,h'):=d(h\cdot x_0,h'\cdot x_0)$, and let $[\cdot,\cdot]_\cdot$ denote the Gromov product with respect to the pullback metric, i.e. $[a,b]_c:=\langle a\cdot x_0,b\cdot x_0 \rangle_{c\cdot x_0}$. We note that Lemma \ref{definition: strong hyperbolicity} is equivalent to the existence of some $L,\varepsilon,A>0$ such that for all $a,b,c,h\in \Gamma$ with 
    \[[ b,c ]_h-[ a,b ]_h:=R\geq A,\]
    we have that
    \begin{equation}\label{equation: exponential convergence of Gromov product at infinity}
        \abs{[a,c]_h - [ a,b]_h}\leq Le^{- R}, 
    \end{equation}
    which we easily derive from Definition \ref{definition: strong hyperbolicity} by expanding the Gromov products and identifying $a\cdot x_0,b\cdot x_0,c\cdot x_0,h\cdot x_0$ with $t,z,x,y$ respectively.
    By definition of the Gromov product, we note that 
    \begin{equation}\label{equation: conjugacy length in terms of gromov products.}
    \begin{split}
                \hat{d}\left(g\overline{\lambda(x)},\overline{\lambda(x)}\right) = &-2\left(\left[ \overline{\lambda(x)},g\overline{\lambda(x)}\right]_{g}+ \left[ g, \overline{\lambda(x)}\right]_e\right) \\ &+ \hat{d}\left(g\overline{\lambda(x)},g\right)+\hat{d}(g,e)+\hat{d}\left(\overline{\lambda(x)},e\right).
    \end{split}
    \end{equation}
    Using the fact that $\hat{d}$ is quasi-isometric to the word metric, we see using \eqref{equation: quasi-invariance} and the fact that $x$ is a geodesic word starting with $u$ that
    \begin{equation*}
        \begin{split}
        \left[ \overline{\lambda(u)},\overline{\lambda(x)}\right]_e\geq \lambda^{-1}\left(\overline{\lambda(u)},\overline{\lambda(x)}\right)_e-B\geq \lambda^{-1}(l-1)-B
        \end{split}
    \end{equation*}
    for some bounded  $\lambda,B>0$. Hence applying \eqref{equation: exponential convergence of Gromov product at infinity} for $a=g$, $b= \overline{\lambda(u)}$ and $c=\overline{\lambda(x)}$ we see that \[\abs{\left[ g, \overline{\lambda(u)}\right]_e-\left[ g, \overline{\lambda(x)}\right]_e}\leq K_1e^{-\epsilon\lambda^{-1}l},\]
    with $K_1$ a constant that can be chosen to be increasing in $[ g, \overline{\lambda(u)}]_e$. Similarly, we can show that \[\left[ g\overline{\lambda(x)},g\overline{\lambda(u)}\right]_g=\left[ \overline{\lambda(x)},\overline{\lambda(u)}\right]_e\geq \lambda^{-1}(l-1)-B\] and hence that 
    \[\abs{\left[ \overline{\lambda(x)},g\overline{\lambda(u)}\right]_g -\left[ \overline{\lambda(x)},g\overline{\lambda(x)}\right]_g}\leq K_2e^{-\epsilon\lambda^{-1}l}\]
    with $K_2$ increasing in $[ \overline{\lambda(x)},g\overline{\lambda(u)}]_g$. In turn, using the estimate \[[ \overline{\lambda(x)},\overline{\lambda(u)}]_g \geq \lambda^{-1}(\overline{\lambda(x)},\overline{\lambda(u)})_g-B,\] we see that
    \[[ \overline{\lambda(x)},\overline{\lambda(u)}]_g\geq \lambda^{-1}(l-1)-B.\] 
    Furthermore $d_{{\Gamma_0}}(g,\overline{\lambda(w)})\geq \abs{w}-1$ for any $w\in \mathfrak{L}(\mathcal{G})$ by the $Z(g)$-Markov property of the Redfern coding. Consequently by another application of \eqref{equation: exponential convergence of Gromov product at infinity}, we see that $[ \overline{\lambda(x)},g\overline{\lambda(u)}]_g$ and hence $[ \overline{\lambda(x)},g\overline{\lambda(x)}]_g$ is exponentially close to $[ \overline{\lambda(u)},g\overline{\lambda(u)}]_g$. Applying our observations on the Gromov products on the right-hand side of \eqref{equation: conjugacy length in terms of gromov products.}, we conclude that 
    \[ \abs{\hat{d}\left(g\overline{\lambda(x)},\overline{\lambda(x)}\right)-2L_x-\tau(u) }\leq Ke^{-\epsilon\lambda^{-1}l},\]
    where $K$ is a constant depending only on $\hat{d}$, $[ \overline{\lambda(u)},g\overline{\lambda(u)}]_g$ and $[ g, \overline{\lambda(u)}]_e$. In fact we see by our proof that $K$ can be bounded as long as these two Gromov products can be bounded.  Bounding $[ g, \overline{\lambda(u)}]_e$ is a straightforward exercise using the $Z(g)$-Markov property (and of course remembering that $g\in Z(g)$). To bound the other Gromov product, we assume by contradiction that there exists a sequence $u_n\in \mathfrak{L}(\mathcal{G})$ such that $[ \overline{\lambda(u_n)},g\overline{\lambda(u_n)}]_g=[ g^{-1}\overline{\lambda(u_n)},\overline{\lambda(u_n)}]_e\to \infty$. By a diagonal argument, we may assume that $u_n\to (v_0,v_1,\ldots) \in \Sigma$ and hence $\overline{\lambda(u_n)}\to \xi$ for some $\xi \in \partial\Gamma$. But the Gromov product tending to $\infty$ implies that $g^{-1}\cdot \xi=\xi$, so $\xi$ is a limit point of the group generated by $g$ and hence of $Z(g)$, which by quasiconvexity of $Z(g)$ in turn implies that $n\mapsto \overline{\lambda(v_0,v_1,\ldots,v_n)}$ is a geodesic ray which stays within a bounded distance of $Z(g)$, contradicting the $Z(g)$-Markov property of the Redfern coding, since
    \[d_{{\Gamma_0}}\left(\overline{\lambda(v_0,v_1,\ldots,v_n)},g\right)\geq \min_{h\in Z(g)}d_{{\Gamma_0}}(\overline{\lambda(v_0,v_1,\ldots,v_n)},h)=n,\]
    where we recall that $\lambda(v_0,v_1,\ldots,v_n)$ has $n$ labelled edges.
\end{proof}

\begin{lemma}\label{Lemma: Convergence of the Sum}
The following limit exists and is finite:
\[\lim_{l\to +\infty} \sum_{\substack{|u|= l\\ u_0= *}}C_{u}e^{-\delta\frac{\tau(u)}{2}} =: C>0,\]
where $\tau$ is the function appearing in Proposition \ref{Proposition: Estimating Lengths}, the constant $\delta$ is the same as in Proposition \ref{proposition: counting representatives of cosets} and we use the convention that $C_{u}=0$ when $\mathcal{J}(u)= \emptyset $.
\end{lemma}
\begin{proof}
By Proposition \ref{proposition: form of Cuv}, if $D_j:=\frac{h_a^{(j)}(\mathbf{0})}{-a\rstr{\frac{d}{dt}P_j(-tr)}{t=a}}$ for $j\in \mathcal{M}$, then
\[\lim_{l\to +\infty} \sum_{\substack{|u|=l\\ u_0= *}}C_{u}e^{-\delta\frac{\tau(u)}{2}} = \lim_{l\to +\infty} \sum_{j\in \mathcal{M}} \sum_{\substack{|u|=l\\ u_0= *}} D_j\mu_a^{(j)}([u])e^{-\delta\frac{\tau(u)}{2}}.\]
Consider for $l\in \mathbb{N}$ and $j\in\mathcal{M}$ the function \[F_{j,l}(x):=\sum_{\substack{|u|=l\\ u_0= *}}  D_j\chi_{[u]}(x)e^{-\delta\frac{\tau(u)}{2}},\]
which we note is constant on cylinders of length $l$. Using the expression for $\tau$ in Proposition \ref{Proposition: Estimating Lengths}, we see that continuity of the Gromov product implies that $\tau$ can be defined on $x=(x_l)_{n=0}^\infty\in\Sigma$ by setting $\zeta_x:=\lim_{n\to \infty}\overline{\lambda((x_0,\ldots,x_n))}$ when $x$ is not a terminating sequence and substituting $\overline{\lambda(u)}$ with $\zeta_x$ in the expression for $\tau(u)$ in the statement of the proposition. Since $F_{j,l}(x)\to D_je^{-\delta\frac{\tau(x)}{2}}$ pointwise on the cylinder $[*]$ as $l\to \infty$ and since the functions $F_{j,l}$ are uniformly bounded by boundedness of $\tau(u)$, we obtain by the dominated convergence theorem that 
\begin{equation*}
\begin{split}
\lim_{l\to\infty}\int_{[*]} F_{j,l}(x)\,d\mu_a^{(j)}(x)&=\lim_{l\to\infty}\sum_{\substack{|u|=l\\ u_0= *}} D_j\mu_a^{(j)}([u])e^{-\delta\frac{\tau(u)}{2}}\\ &=\int_{[*]}  D_je^{-\delta\frac{\tau(x)}{2}}\,d\mu_a^{(j)}(x).
\end{split}
\end{equation*}
Hence
\[\lim_{l\to +\infty} \sum_{\substack{|u|=l\\ u_0= *}}C_{u}e^{-\delta\frac{\tau(u)}{2}} =\sum_{j\in \mathcal{M}}\int_{[*]}  D_je^{-\delta\frac{\tau(x)}{2}}\,d\mu_a^{(j)}(x),\]
which is positive since all maximal components are reachable from the distinguished vertex $*$. 
\end{proof}
\begin{proof}[Proof of Theorem \ref{Hypothesis:First Step}]
For notational simplicity, let 
\[N(T):=\# \{g' \in \mathrm{Cl}(g) \hbox{ : }  d(x_0,g' \cdot x_0)\leq T\}.\]

Let us fix $\mathfrak{k}>0$ and write $\epsilon_l:= K\rho^l$. By the discussion in the introduction and at the start of Section \ref{section: Counting Representatives of Cosets}, we have that every element of $\mathrm{Cl}(g)$ is equal to $\overline{\lambda(x)}^{-1}g\overline{\lambda(x)}$ for a unique $x\in \mathfrak{L}(\mathcal{G})$.  By Proposition \ref{Proposition: Estimating Lengths} we have that
\begin{equation*}
    \begin{split}
        &N^{u}\left(\frac{T-\tau(u)+\epsilon_l}{2}\right) \\
        &\geq \#\left\{x\in J(u):d\left(\overline{\lambda(x)}^{-1}g\overline{\lambda(x)}\cdot x_0,x_0\right)\leq T\right\}.
    \end{split}
\end{equation*}
Hence we see that 
\begin{equation*}
        N(T) \leq \sum_{\substack{|u|= l\\ u_0= *}}N^{u}\left( \frac{T-\tau(u)+\epsilon_l}{2} \right).
\end{equation*}
Using Proposition \ref{proposition: counting representatives of cosets}, we have that there exists some $M(l,\mathfrak{k})\geq 0$ depending on $l$ and $\mathfrak{k}$ such that 
\[N(T)\leq \left((1+\mathfrak{k})e^{\delta\epsilon_l/2} \sum_{\substack{|u| = l\\ u_0=*}} C_{u}e^{-\delta\frac{\tau(u)}{2}}\right)e^{\delta T/2},\]
for $T\geq M(l,\mathfrak{k})$. By Lemma \ref{Lemma: Convergence of the Sum} there exists some $K(\mathfrak{k})\geq0$ depending on $\mathfrak{k}$ such that
\[\sum_{\substack{|u| = l\\ u_0=*}}C_{u}e^{-\delta\frac{\tau(u)}{2}}\leq(1+\mathfrak{k})C\text{ and } \delta\epsilon_l\leq \mathfrak{k}\]
when $l\geq K(\mathfrak{k})$. 
In particular for $T\geq M\left(K(\mathfrak{k}),\mathfrak{k}\right)$, we have that 
\[N(T)\leq \left((1+\mathfrak{k})^2e^{\mathfrak{k}/2} C\right)e^{\delta
T/2}.\]
Arguing analogously, we may obtain the lower bound 
\[N(T)\geq \left((1-\mathfrak{k})^2e^{-\mathfrak{k}/2} C\right)e^{\delta T/2}\]
for sufficiently large values of $T$ and $l$. We obtain Theorem \ref{Hypothesis:First Step} by letting $\mathfrak{k}\to 0$. 

\end{proof}

\appendix
\section{Lexicographical Ordering and Automata}

We can impose a lexicographical ordering $\prec$ on $S^*$ by saying two words $w=a_1\cdots a_k, w'=a_1'\cdots a_{k'}'\in S^*$ of different length satisfy the ordering $w \prec w'$ if $k=\abs{w}<\abs{w'}=k'$. If they have the same length, but $w\neq w'$, then we consider the first value of $i\in \{1,2,\ldots,k'\}$ for which $a_i\neq a_i'$ and let $w\prec w'$ if $a_i< a_i'$ as natural numbers.

Using this lexicographical ordering, we obtain the following choice of representations of $\Gamma$.
\begin{definition}
    Given a group $\Gamma$, a symmetric choice of generators ${\Gamma_0}$, and an ordering $S\mapsto {\Gamma_0}$, we define $\lex$ to be the set of all elements $w\in S^*$ for which the conditions 
    \[ w\neq w'\in S^* \text{ and } \overline{w'}=\overline{w} \]
    imply that $w\prec w'$ with respect to the lexicographical ordering discussed previously in this subsection. 
\end{definition}
We note that the word \enquote{automaton} has several possible equivalent definitions. One possible way of defining it is as a finite directed graph $(\mathcal{G})$ with distinguished vertex $*$ and a labelling $\lambda$ of edges as in Definition \ref{definition: path-labelling}, with the addition of a subset $V_{\mathrm{accept}}\subset V$ of the vertex set being designated \enquote{accept states}. The \textbf{language} $\mathfrak{L}(\mathcal{G})$ of $\mathcal{G}$ is then defined to be the set of all images of paths under the path-labelling map $\lambda$ starting in $*$ and ending at a vertex in $V_{\mathrm{accept}}$. We refer to the first chapter of \cite{Epstein92} to see that this definition is equivalent to other definitions of automata. 

Definition \ref{definition: path-labelling} is then equivalent to all the vertices being accept states, which is equivalent to saying the language is \textbf{prefix closed}, where we again refer to \cite{Epstein92} for the definition. We see that having the $H$-Markov property is then equivalent to the existence of the following:
\begin{itemize}
    \item automaton $\mathcal{G}$ generating a prefix-closed language $\mathfrak{L}(\mathcal{G})$ such that 
    \item $\lambda\left(\mathfrak{L}(\mathcal{G})\right)\subset \lex$, and 
    \item each word in $\mathfrak{L}(\mathcal{G})$ corresponds under the natural map to a unique right coset in $H\backslash\Gamma$.
\end{itemize}
That such an automaton exists follows from the construction in \cite{Redfern93}. This is also explicitly mentioned in \cite{Holt99}.

\section{CAT(-1)-spaces}
One natural way to generalise a convex cocompact action of a group $\Gamma$ acting on a simply connected Riemannian manifold of negative curvature $X$ is to replace $X$ with a $\CAT$ space.

If $\mathbf{\Delta}$ is a triangle with vertices $x,y,z$, we define a \textbf{comparison triangle} to be a triangle $\mathbf{\Delta}'\subset \mathbf{H}$ with vertices $\overline{x},\overline{y},\overline{z}$ such that $d_X(x,y)=d_\mathbf{H}(\overline{x},\overline{y})$, $d_X(y,z)=d_\mathbf{H}(\overline{y},\overline{z})$ and $d_X(z,x)=d_\mathbf{H}(\overline{z},\overline{x})$, where $d_\mathbf{H}$ is the hyperbolic metric on $\mathbf{H}$. We note that comparison triangles are unique up to hyperbolic isometry. There is a unique bijective map $F:\mathbf{\Delta}\to \mathbf{\Delta}'$ which sends $x,y,z$ to $\overline{x},\overline{y},\overline{z}$ respectively and which is an isometry when restricted to any one side. 
\begin{definition}
    A geodesic metric space is a $\CAT$ space if for each triangle $\mathbf{\Delta}$ and any two points $a,b\in \mathbf{\Delta}$, 
    \[d_X(a,b)\leq d_\mathbf{H}(F(a),F(b)),\]
    where $F$ is the map to the comparison triangle as defined above.
\end{definition}
Examples of $\CAT$ spaces include Riemannian manifolds of negative sectional curvature $\leq -1$ and metric trees. The proof that the metric on $\CAT$-spaces satisfies Lemma \ref{definition: strong hyperbolicity} can be found in \cite{NSpak}, although we note that a similar condition from which Hölder continuity of the roof follows can be found in \cite{PolSharp01}.
\printbibliography

\Addresses

\end{document}